\documentclass[12pt]{amsart}

\usepackage{amsmath, amssymb}

\newtheorem{theorem}{Theorem}[section]
\newtheorem{lemma}[theorem]{Lemma}

\newtheorem{cor}[theorem]{Corollary}

\theoremstyle{definition}

\theoremstyle{remark}
\newtheorem{remark}[theorem]{Remark}

\numberwithin{equation}{section}

\newcommand{\rr}{{\mathbb R}}
\newcommand{\rd}{{\mathbb R^d}}

\newcommand{\nat}{{\mathbb N}}

\newcommand{\Exp}{{\mathbb E}}
\newcommand{\Ker}{\operatorname{Ker}}
\newcommand{\GL}{\operatorname{GL}}

\allowdisplaybreaks

\begin{document}
\sloppy
\title[Hausdorff dimension of operator semistable L\'evy processes]{The Hausdorff dimension of\\ operator semistable L\'evy processes} 
\author{Peter Kern}
\address{Peter Kern, Mathematisches Institut, Heinrich-Heine-Universit\"at D\"usseldorf, Universit\"atsstr. 1, D-40225 D\"usseldorf, Germany}
\email{kern\@@{}math.uni-duesseldorf.de}

\author{Lina Wedrich}
\address{Lina Wedrich, Fakult\"at f\"ur Wirtschaftswissenschaften, Universit\"at Duisburg-Essen, Campus Essen, Universit\"atsstr. 12, D-45141 Essen, Germany}
\email{lina.wedrich\@@{}uni-due.de} 

\date{\today}

\begin{abstract}
Let $X=\{X(t)\}_{t\geq0}$ be an operator semistable L\'evy process in $\rd$ with exponent $E$, where $E$ is an invertible linear operator on $\rd$ and $X$ is semi-selfsimilar with respect to $E$. By refining arguments given in Meerschaert and Xiao \cite{MX} for the special case of an operator stable (selfsimilar) L\'evy process, for an arbitrary Borel set $B\subseteq\rr_+$ we determine the Hausdorff dimension of the partial range $X(B)$ in terms of the real parts of the eigenvalues of $E$ and the Hausdorff dimension of $B$.
\end{abstract}

\keywords{L\'evy process, operator semistable process, semi-selfsimilarity, sojourn time, range, Hausdorff dimension, positivity of density}
\subjclass[2010]{Primary 60G51; Secondary 28A78, 28A80, 60G17, 60G52.}

\maketitle

\baselineskip=18pt

\section{Introduction}

Let $X=\{X(t)\}_{t\geq0}$ be a L\'evy process in $\rd$, i.e. a stochastically continuous process with stationary and independent increments, starting in the origin $X(0)=0$ almost surely. Without loss of generality, we will assume that the process has c\`adl\`ag paths (right continuous with left limits). The distribution of the process on the space of  c\`adl\`ag functions is uniquely determined by the distribution of $X(1)$ which can be an arbitrary infinitely divisible distribution. We will always assume that the distribution of $X(1)$ is full, i.e. not supported on any lower dimensional hyperplane. The L\'evy process $X$ is called operator semistable if the distribution $\mu_1=P_{X(1)}$ is strictly operator semistable, i.e. $\mu_1$ is an infinitely divisible probability measure fulfilling
\begin{equation}\label{sos}
\mu_1^{\ast c}=c^E\mu_1\quad\text{ for some }c>1
\end{equation}
and some linear operator $E$ on $\rd$ called the exponent, where $c^E\mu_1(dx)=\mu_1(c^{-E}dx)$ denotes the image measure under the invertible linear operator $c^E=\sum_{n=0}^\infty\frac{(\log c)^n}{n!}\,E^n$. For details on operator semistable distributions we refer to \cite{Luc,Cho} and the monograph \cite{MS}. To be more precise, we call the L\'evy process $(c^E,c)$-operator semistable due to the space-time scaling
\begin{equation}\label{soss}
\{c^EX(t)\}_{t\geq0}\stackrel{\rm fd}{=}\{X(ct)\}_{t\geq0}\quad\text{ for some }c>1
\end{equation}
which easily follows from \eqref{sos}, where $\stackrel{\rm fd}{=}$ denotes equality of all finite dimensional marginal distributions. The property \eqref{soss} is called strict operator semi-selfsimilarity and one can equivalently introduce an operator semistable L\'evy process as a strictly operator semi-selfsimilar L\'evy process. It is well known that for a given operator semistable L\'evy process $X$ the exponent $E$ is not unique, but the real parts of the eigenvalues of every possible exponent are the same, including their multiplicity; see \cite{MS}.

In case \eqref{sos} or, equivalently, \eqref{soss} is fulfilled for every $c>0$ the L\'evy process is called operator stable, respectively strict operator selfsimilar, with exponent $E$. In the last decades efforts have been made to calculate the Hausdorff dimension of the range $X([0,1])$ for an operator stable L\'evy process $X$. For a survey on general dimension results for L\'evy processes we refer to \cite{Xiao, KX}. If $X$ is an $\alpha$-stable L\'evy process in $\rd$ for some $\alpha\in(0,2]$, i.e. the exponent is a multiple of the identity $E=\alpha\cdot I$, Blumenthal and Getoor \cite{BG} show that the Hausdorff dimension of the range is $\dim_{\rm H}X([0,1])=\min(\alpha,d)$ almost surely. Pruitt and Taylor \cite{PT} calculate $\dim_{\rm H}X([0,1])$ for a L\'evy process in $\rd$ with independent stable marginals of index $\alpha_1\geq\cdots\geq\alpha_d$. Here, $\dim_{\rm H}X([0,1])=\alpha_1$ almost surely if $\alpha_1\leq1$ or $\alpha_1=\alpha_2$ and in all other cases $\dim_{\rm H}X([0,1])=1+\alpha_2(1-\alpha_1^{-1})\in(\alpha_2,\alpha_1)$ almost surely. In this case $E$ is a diagonal operator with $\alpha_1,\ldots,\alpha_d$ on the diagonal in a certain order. Later, based on results of Pruitt \cite{Pru}, Becker-Kern, Meerschaert and Scheffler \cite{BMS} obtained that for more general operator stable L\'evy processes the formulas of Pruitt and Taylor are still valid without the assumption of independent stable marginals, where $\alpha_1,\ldots,\alpha_d$ have to be interpreted as the reciprocals of the real parts of the eigenvalues of the exponent $E$. Their result does not cover the full class of operator stable L\'evy processes, since in case $\alpha_1>\min(1,\alpha_2)$ it is required that the density of $X(1)$ is positive at the origin. Finally, Meerschaert and Xiao \cite{MX} show that the restriction on the density is superflous. In addition they calculate the Hausdorff dimension of the partial range $\dim_{\rm H}X(B)$ for an arbitrary operator stable L\'evy process $X$ and an arbitrary Borel set $B\subseteq\rr_+$ in terms of the real parts of the eigenvalues of the exponent $E$ and the Hausdorff dimension of $B$, namely
\begin{equation}\label{MXmain}
\dim_{\rm H}X(B)=\begin{cases}\alpha_1\dim_{\rm H}B & \text{ if } \alpha_1\dim_{\rm H}B\leq1\text{ or }\alpha_1=\alpha_2,\\ 1+\alpha_2\big(\dim_{\rm H}B -\frac1{\alpha_1}\big) & \text{ otherwise.}
\end{cases}
\end{equation}

Since operator semistable L\'evy processes require the space-time scaling property to be only fulfilled on a discrete scale, they allow more flexibility in modeling. The most prominent example of a semistable, non-stable distribution is perhaps the limit distribution of cumulative gains in a series of St.~Petersburg games. Our aim is to generalize the above dimension results for the larger class of operator semistable L\'evy processes, following the outline given by \cite{MX}. We will prove that \eqref{MXmain} remains valid for operator semistable L\'evy processes, but our methods go beyond simple adjustments of the arguments given in \cite{MX}. To the best of our knowledge, our result is the first dimension result for L\'evy processes with a scaling or selfsimilarity property on a discrete scale. Whereas, for deterministic selfsimilar sets (on a discrete scale), numerous examples for a determination of the Hausdorff dimension and other fractal dimensions exist in the literature, e.g. for Cantor sets or Sierpinski gaskets.

The paper is organized as follows. In section 2.1 we recall the definitions of Hausdorff and capacitary dimension and their relationship. We further recall a spectral decomposition result from \cite{MS} in section 2.2, which enables us to decompose the operator semistable L\'evy process according to the distinct real parts of the eigenvalues of the exponent $E$. Preparatory for the proof of our main results, in section 2.3 certain uniform density bounds for $\{X(t)\}_{t\in[1,c)}$ are given and a certain positivity set for the densities is constructed. These will be needed to obtain sharp lower bounds for the expected sojourn times of operator semistable L\'evy processes in a closed ball in section 2.4. Note that the characterization of the positivity set of densities is still an open problem even for operator stable densities. In the special case of an $\alpha$-stable L\'evy process with exponent $E=\alpha\cdot I$ the problem is completely solved in a series of papers \cite{Tay,Port, PV, ARRS} . A certain extension for $\alpha$-semistable L\'evy processes can be found in section 3 of \cite{SW}. Finally, in section 3 we state our main results on the Hausdorff dimension of operator semistable sample paths, including the proofs.

Throughout this paper $K$ denotes an unspecified positive and finite constant which may vary in each occurrence. Specified constants will be denoted by $K_1,K_2$, etc.

\section{Preliminaries}

\subsection{Hausdorff and capacitary dimension} 

For an arbitrary subset $A\subseteq\rd$ and $s\geq0$  the $s$-dimensional Hausdorff measure is defined by
\begin{equation}\label{sdimH}
	\mathcal{H}^{s}(A)=\lim_{\varepsilon\downarrow0}\inf\left\{\sum_{i=1}^{\infty}|A_{i}|^{s}:\,A\subseteq\bigcup_{i=1}^{\infty}A_{i},\; 0<|A_{i}|\leq \varepsilon\right\},
\end{equation}
where $|A|=\sup\{\|x-y\|:\,x,y\in A\}$ denotes the diameter of $A\subseteq\rd$. The sequence of sets $\{A_{i}\}_{i\geq1}$ fulfilling the conditions on the right-hand side of \eqref{sdimH} is called an $\varepsilon$-covering of $A$. It can be shown that $\mathcal{H}^{s}$ is a metric outer measure on $\rd$ and there exists a unique value $\dim_{\rm H}A\geq0$ such that $\mathcal{H}^{s}(A)=\infty$ if  $0\leq s<\dim_{\rm H}A$ and $\mathcal{H}^{s}(A)=0$ if $\dim_{\rm H}A<s<\infty$; e.g., see \cite{Fal1,Fal2}. The critical value
\begin{equation}\label{Hdim}
\dim_{\rm H}A=\inf\{s>0:\,\mathcal{H}^{s}(A)=0\}=\sup\{s>0:\,\mathcal{H}^{s}(A)=\infty\}
\end{equation}
is called the Hausdorff dimension of $A$.

Now let $A\subseteq\rd$ be a Borel set and denote by $\mathcal M^1(A)$ the set of probability measures on $A$. For $s>0$ the $s$-energy of $\mu\in\mathcal M^1(A)$ is defined by
$$I_s(\mu)=\int_A\int_A\frac{\mu(dx)\mu(dy)}{\|x-y\|^s}.$$
By Frostman's lemma, e.g., see \cite{Kah,Mat}, there exists a probability measure $\mu\in\mathcal M^1(A)$ with $I_s(\mu)<\infty$ if $\dim_{\rm H}A>s$. In this case $A$ is said to have positive $s$-capacity $C_s(A)$ given by
$$C_s(A)=\sup\{I_s(\mu)^{-1}:\,\mu\in\mathcal M^1(A)\}$$
and the capacitary dimension of $A$ is defined by
$$\dim_{\rm C}A=\sup\{s>0:\,C_s(A)>0\}=\inf\{s>0:\,C_s(A)=0\}.$$
A consequence of Frostman's theorem, e.g., see \cite{Kah,Mat}, is that for Borel sets $A\subseteq\rd$ the Hausdorff and capacitary dimension coincide. Therefore, one can prove lower bounds for the Hausdorff dimension with a simple capacity argument: if $I_s(\mu)<\infty$ for some $\mu\in\mathcal M^1(A)$ then $\dim_{\rm H}A=\dim_{\rm C}A\geq s$.

\subsection{Spectral decomposition}

Let $\{X(t)\}_{t\geq0}$ be a $(c^{E},c)$-operator semistable L\'evy process in $\rd$. Factor the minimal polynomial of $E$ into $f_{1}(x)\cdot\ldots\cdot f_{p}(x)$ such that every root of $f_{j}$ has real part $a_{j}$, where $a_1<\cdots<a_p$ are the distinct real parts of the eigenvalues of $E$ and $a_1\geq\frac12$ by Theorem 7.1.10 in  \cite{MS}. According to  Theorem 2.1.14 in \cite{MS} we can decompose $\rd$ into a direct sum $\rd=V_{1}\oplus\ldots\oplus V_{p}$, where $V_{j}=\Ker(f_{j}(E))$ are $E$-invariant subspaces. Now, in an appropriate basis, $E$ can be represented as a block-diagonal matrix $E=E_{1}\oplus\ldots\oplus E_{p}$, where $E_{j}:V_{j}\rightarrow V_{j}$ and every eigenvalue of $E_{j}$ has real part $a_{j}$. Especially, every $V_j$ is an $E_j$-invariant subspace of dimension $d_j=\dim V_j$. Now we can write $x=x_1+\cdots+x_p\in\rd$ and $t^Ex=t^{E_1}x_1+\cdots+t^{E_p}x_p$ with respect to this direct sum decomposition, where $x_j\in V_j$ and $t>0$. Moreover, for the operator semistable L\'evy process we have $X(t)=X^{(1)}(t)+\ldots+X^{(p)}(t)$ with respect to this direct sum decomposition, where $\{X^{(j)}(t)\}_{t\geq0}$ is a  $(c^{E_{j}},c)$-operator semistable L\'evy process on $V_j\cong \rr^{d_j}$ by  Lemma 7.1.17 in \cite{MS}. We can further choose an inner product on $\rd$ such that the subspaces $V_j$, $1\leq j\leq p$, are mutually orthogonal and throughout this paper for $x\in\rd$ we may choose $\|x\|=\langle x,x\rangle^{1/2}$ as the associated Euclidean norm on $\rd$. With this choice, in particular we have for $t=c^rm>0$
\begin{equation}
	\|X(t)\|^2\stackrel{\rm d}{=}\|c^{rE}X(m)\|^{2}=\|c^{rE_{1}}X^{(1)}(m)\|^{2}+\ldots+\|c^{rE_{p}}X^{(p)}(m)\|^{2},
\end{equation}
with $r\in\mathbb{Z}$ and $m\in[1,c)$.
The following result on the growth behavior of the exponential operators $t^{E_j}$ near the origin $t=0$ is a reformulation of Lemma 2.1 in \cite{MX} and a direct consequence of Corollary 2.2.5 in \cite{MS}.

\begin{lemma}\label{specbound}
For every $j=1,\ldots,p$ und every $\varepsilon>0$ there exists a finite constant $K\geq 1$ such that for all $0<t\leq1$ we have
\begin{equation}
	K^{-1}t^{a_{j}+\varepsilon}\leq \|t^{E_{j}}\|\leq K\,t^{a_{j}-\varepsilon}
\end{equation}
and
\begin{equation}
	K^{-1}t^{-(a_{j}-\varepsilon)}\leq \|t^{-E_{j}}\|\leq K\,t^{-(a_{j}+\varepsilon)}.
\end{equation}
\end{lemma}

Throughout this paper let $\alpha_j=1/a_j$ denote the reciprocals of the distinct real parts of the eigenvalues of $E$ with $0<\alpha_p<\cdots<\alpha_1\leq2$.

\subsection{Density bounds}

Let $X=\{X(t)\}_{t\geq0}$ be an operator semistable L\'evy process in $\rd$ with $P_{X(t)}=\mu_t$ for $t>0$. It is well known that integrability properties of the Fourier transform $\widehat{\mu_t}$ imply the existence and certain smoothness properties of a Lebesgue density  of $\mu_t$.  In fact, $|\widehat{\mu_t}|$ has at least exponential decay in radial directions for every $t>0$, i.e.
\begin{equation}\label{expdecay}
\left|\widehat{\mu_t}(x)\right|=\left|\widehat{\mu_1}(x)\right|^t\leq\exp\left(-tK\|x\|^{1/m}\right)\quad\text{ if }\|x\|>M,
\end{equation}
where $m\in\nat$, $M>0$, $K>0$ are certain constants not depending on $t$. For an operator semistable L\'evy process without Gaussian component (i.e. $\alpha_1<2$) this follows directly from equation (2.4) in \cite{Luc}. In case $\alpha_1=2$ the spectral component $X^{(1)}(t)$ has a centered Gaussian distribution with positive definite covariance matrix $\Sigma=R^\top R$ according to fullness. Hence
\begin{align*}
\widehat{P_{X^{(1)}(t)}}(x_1) & =\exp\left(-\frac12\, t\|Rx_1\|^2\right)\leq\exp\left(-\frac1{2\|R^{-1}\|}\, t\|x_1\|^2\right)\\
& =\exp\left(-t\,C_1\|x_1\|^2\right).
\end{align*}
By the L\'evy-Khintchine representation, $X^{(1)}(t)$ is independent of $X^{(2)}(t)+\cdots+X^{(p)}(t)$ and together with equation (2.4) in \cite{Luc} we get for $\|x\|>M\geq1$
\begin{align*}
\left|\widehat{\mu_t}(x)\right| & \leq\exp\left(-t\,C_1\|x_1\|^2\right)\cdot\exp\left(-t\,C_2\|x_2+\cdots+x_p\|^{1/m}\right)\\
& =\exp\left(-t\,C_1(\|x_1\|^2)^{1/(2m)}\right)\cdot\exp\left(-t\,C_2\Big(\sum_{j=2}^p\|x_j\|^2\Big)^{1/(2m)}\right)\\
& \leq\exp\left(-tK\|x\|^{1/m}\right),
\end{align*}
where $K=\min(C_1,C_2)$. Thus we have also shown \eqref{expdecay} in case $X(t)$ has a Gaussian component.
According to Proposition 28.1 in \cite{Sat},  for every $t>0$ the random vector $X(t)$ has a Lebesgue density $x\mapsto g_t(x)$ of class $C^\infty(\rd)$ and $g_t(x)\to0$ as $\|x\|\to\infty$. We will additionally need certain uniformity results for the densities.

\begin{lemma}\label{densbound}
The mapping $(t,x)\mapsto g_t(x)$ is continuous on $(0,\infty)\times\rd$ and we have
\begin{equation}\label{supsup}
\sup_{t\in[1,c)}\sup_{x\in\rd}\left|g_t(x)\right|<\infty.
\end{equation}
\end{lemma}

\begin{proof}
For any sequence $(t_n,x_n)\to(t,x)$ in $(0,\infty)\times\rd$ by Fourier inversion and dominated convergence we have
$$g_{t_n}(x_n)=(2\pi)^{-d}\int_{\rd}e^{-i\langle x_n,y\rangle}\widehat{\mu_{t_n}}(y)\,d\lambda^d(y)\to(2\pi)^{-d}\int_{\rd}e^{-i\langle x,y\rangle}\widehat{\mu_{t}}(y)\,d\lambda^d(y)=g_t(x),$$
where $\lambda^d$ denotes Lebesgue measure on $\rd$. This shows continuity of  $(t,x)\mapsto g_t(x)$. Moreover, $\|g_t\|_\infty=\sup_{x\in\rd}\left|g_t(x)\right|$ is continuous in $t>0$, hence \eqref{supsup} follows.
\end{proof}

Consequently, we get a refinement of Lemma 3.1 in \cite{BMS} on the existence of negative moments of an operator semistable L\'evy process $X=\{X(t)\}_{t\geq0}$ in $\rd$.

\begin{lemma}\label{negmom}
For any $\delta\in(0,d)$ we have
$$\sup_{t\in[1,c)}\Exp\left[\|X(t)\|^{-\delta}\right]<\infty.$$
\end{lemma}

\begin{proof}
Let $g_t$ be as before and define $K=\sup_{t\in[1,c)}\sup_{x\in\rd}\left|g_t(x)\right|$, then $K<\infty$ by Lemma \ref{densbound}. In view of $\delta<d$ we have for every $t\in[1,c)$
\begin{align*}
\Exp\left[\|X(t)\|^{-\delta}\right] & =\int_\rd\|x\|^{-\delta}g_t(x)\,dx\\
& \leq K\int_{\{\|x\|\leq 1\}}\|x\|^{-\delta}\,dx+\int_{\{\|x\|>1\}}g_t(x)\,dx\\
& \leq K\int_{\{\|x\|\leq 1\}}\|x\|^{-\delta}\,dx+1<\infty.
\end{align*}
Since this upper bound is independent of $t\in[1,c)$, the assertion follows.
\end{proof}

By a result of Sharpe \cite{Sha},  for a one-dimensional $(c^{1/\alpha},c)$-semistable L\'evy process we can further deduce from Lemma \ref{densbound} that the positivity set $A_t=\{x\in\rr:\,g_t(x)>0\}$ is either the whole real line $\rr$ or a half line $(at,\infty)$ or $(-\infty,at)$ for some $a\in\rr$ and for all $t>0$ . We will now use a similar argument as given on page 83 in \cite{ARRS} to show that in case $\alpha>1$ we have $g_t(0)>0$. If $A_t=\rr$ there is nothing to prove. Suppose that $A_t=(at,\infty)$ for some $a\geq0$. Let $(Y_n)_{n\in\nat}$ be an i.i.d. sequence with $Y_1\stackrel{\rm d}{=}X(t)$. Since $\alpha>1$ we have $\Exp[|Y_1|]<\infty$ and from the strong law of large numbers it follows that for every sequence of positive integers $k_n\to\infty$ we have
\begin{equation}\label{ndoa}
k_n^{-1/\alpha}\sum_{j=1}^{k_n}Y_j\geq k_n^{-1/\alpha}\sum_{j=1}^{\lfloor k_n^{1/\alpha}\rfloor}Y_j\to\Exp[Y_1]=\Exp[X(t)]
\end{equation}
almost surely. On the other hand, since $X(t)$ belongs to its own domain of normal attraction, for $k_n=\lfloor c^n\rfloor$ the left-hand side of \eqref{ndoa} converges in distribution to $X(t)$. It follows that $X(t)\geq\Exp[X(t)]$ almost surely, thus $X(t)=\Exp[X(t)]$ almost surely in contradiction to the fullness of $X(t)$. Hence we must have $a<0$ which implies $g_t(0)>0$. Similarly, the assumption $A_t=(-\infty,at)$ for some $a\leq0$ leads to $X(t)\leq\Exp[X(t)]$ almost surely and again contradicts the fullness of $X(t)$, hence $a>0$ which again implies $g_t(0)>0$. Alltogether we have shown that a bounded continuous density of a $(c^{1/\alpha},c)$-semistable L\'evy process with $\alpha>1$ is of type A; cf. Taylor \cite{Tay}. In the sequel we will  need a more general positivity result for a bounded continuous density of certain operator semistable L\'evy processes.

\begin{lemma}\label{positive}
Let $\{X(t)\}_{t\geq0}$ be an operator semistable L\'evy process with $\alpha_1>1$, $d_1=1$ and with density $g_t$ as above. Then there exist constants $K>0$, $r>0$ and uniformly bounded Borel sets $J_t\subseteq\rr^{d-1}\cong V_2\oplus\cdots\oplus V_p$ for $t\in[1,c)$ such that
$$g_t(x_1,\ldots,x_p)\geq K>0\quad\text{ for all }(x_1,\ldots,x_p)\in[-r,r]\times J_t.$$
Further, we can choose $\{J_t\}_{t\in[1,c)}$ such that  $\lambda^{d-1}(J_t)\geq R>0$ for every $t\in[1,c)$. Note that the constants $K,r$ and $R$ do not depend on $t\in[1,c)$. 
\end{lemma}

\begin{proof}
As argued above, $(t, x_1)\mapsto g_t(x_1)=\int_{\rr^{d-1}}g_t(x_1,\ldots,x_p)\,d\lambda^{d-1}(x_2,\ldots,x_p)$ is continuous and positive in $x_1=0$ for every $t>0$, hence $\min_{t\in[1,c]}g_t(x_1)>0$ for $x_1=0$. Choose $\delta>0$ and $r>0$ such that $g_t(x_1)\geq\delta$ for every $x_1\in[-r,r]$ and $t\in[1,c]$. We will now show that we can choose $K\in(0,\delta)$ and $R>0$ such that for every $t\in[1,c)$ the Borel set 
$$J_t=\left\{(x_2,\ldots,x_p)\in\rr^{d-1}:\,g_t(x_1,\ldots,x_p)\geq K\text{ for every }x_1\in[-r,r]\right\}$$
fulfills $\lambda^{d-1}(J_t)\geq R$. Assume this choice is not possible. Then for every $K\in(0,\delta)$ and $R>0$ there exists $t=t(K,R)\in[1,c)$ such that $\lambda^{d-1}(J_t)<R$. Letting $K\downarrow0$ and $R\downarrow0$, there exists a subsequence such that $t(K,R)\to t_0\in[1,c]$ along this subsequence and we have $g_{t_0}(x_1,\ldots,x_p)=0$ for some $x_1\in[-r,r]$ and Lebesgue almost every $(x_2,\ldots,x_p)\in\rr^{d-1}$. It follows that 
$g_{t_0}(x_1)=0$ in contradiction to $g_{t_0}(x_1)\geq\delta$. It remains to prove that $\{J_t\}_{t\in[1,c)}$ is uniformly bounded. First note that by Fourier inversion for $t_n\to t>0$ we have
\begin{align*}
\left|g_{t_n}(x)-g_t(x)\right| & =(2\pi)^{-d}\left|\int_{\rd}e^{-i\langle x,y\rangle}\left(\widehat\mu(y)^{t_n}-\widehat\mu(y)^{t}\right)\,d\lambda^d(y)\right|\\
& \leq (2\pi)^{-d}\int_{\rd}\left|1-\widehat\mu(y)^{|t_n-t|}\right|\,d\lambda^d(y)\to0
\end{align*}
uniformly in $x\in\rd$, since the upper bound does not depend on $x$. Now assume that $\{J_t\}_{t\in[1,c)}$ is not uniformly bounded. Then for every $n\in\nat$ there exists $t_n\in[1,c)$ such that for some $(x_2^{(n)},\ldots,x_p^{(n)})\in\rr^{d-1}$ with $\|(x_2^{(n)},\ldots,x_p^{(n)})\|\geq n$ we have
$$g_{t_n}\big(x_1,x_2^{(n)},\ldots,x_p^{(n)}\big)\geq K\quad\text{ for every }x_1\in[-r,r].$$
Now choose a subsequence $t_n\to t_0\in[1,c]$ and choose $n\in\nat$ large enough so that $\left|g_{t_n}(x)-g_{t_0}(x)\right|\leq K/2$ for every $x\in\rd$. Then we get along this subsequence
$$g_{t_0}\big(0,x_2^{(n)},\ldots,x_p^{(n)}\big)\geq K/2$$
which contradicts $g_{t_0}(x)\to0$ for $\|x\|\to \infty$ and concludes the proof.
\end{proof}

\subsection{Bounds for the sojourn time}

Let $K_{1}>0$ be a fixed constant. A family $\Lambda(a)$ of cubes of side $a$ in $\rd$ is called $K_1$-nested if no ball of radius $a$ in $\rd$ can intersect more than $K_{1}$ cubes of $\Lambda(a)$. In the sequel we will choose $\Lambda(a)$ to be the family of all cubes in $\rd$ of the form $[k_{1}a, (k_{1}+1)a]\times\cdots\times[k_{d}a, (k_{d}+1)a]$ with $(k_{1},\ldots,k_{d})\in\mathbb{Z}^{d}$. Obviously, this family $\Lambda(a)$ is $3^d$-nested. Let
\begin{align*}
	T(a,s)=\int_{0}^{s}1_{B(0,a)}(X(t))\,dt
\end{align*}
be the sojourn time of the L\'evy process $X=\{X(t)\}_{t\geq0}$ up to time $s>0$ in the closed ball $B(0,a)$ with radius $a$ centered at the origin. The following remarkable covering lemma is due to Pruitt und Taylor \cite[Lemma 6.1]{PT}.

\begin{lemma}\label{PT}
Let $X=\{X(t)\}_{t\geq0}$ be a  L\'evy process in $\rd$ and let $\Lambda(a)$ be a fixed $K_{1}$-nested family of cubes in $\rd$ of side $a$ with $0<a\leq1$. For any $u\geq0$ let $M_{u}(a,s)$ be the number of cubes in $\Lambda(a)$ hit by $X(t)$ at some time $t\in [u,u+s]$. Then
$$\Exp\left[M_{u}(a,s)\right]\leq 2\,K_{1}s\cdot\left(\Exp\left[T\left(\tfrac{a}{3},s\right)\right]\right)^{-1}.$$
\end{lemma}

We now determine sharp upper and lower bounds for the expected sojourn times $\Exp[T(a,s)]$ of an operator semistable L\'evy process. Our proof follows the outline given in \cite[Lemma 3.4]{MX} for the special case of operator stable L\'evy processes, but in our more general situation the estimations are more delicate. Although we only need the lower bounds in this paper, for completeness we also include the upper bounds which might be useful elsewhere, e.g. for studying exact Hausdorff measure functions. Recall the spectral decomposition of Section 2.2 for the constants $\alpha_1,\alpha_2$ and $d_1$ appearing in the following result.

\begin{theorem}\label{sojournbounds}
Let $X=\{X(t)\}_{t\geq0}$ be an operator semistable L\'evy process in $\rd$ with $d\geq2$. For any $0<\alpha_{2}'<\alpha_{2}<\alpha_{2}''<\alpha_{1}'<\alpha_{1}<\alpha_{1}''$ there exist positive and finite constants $K_{2},\ldots,K_{5}$ such that
\begin{itemize}
	\item[(i)] if $\alpha_{1} \leq d_{1}$, then for all $0<a\leq 1$ and $a^{\alpha_{1}}\leq s\leq1$ we have
		\begin{equation*}
			K_{2}a^{\alpha_{1}''}\leq\Exp[T(a,s)]\leq K_{3}a^{\alpha_{1}'}.
		\end{equation*}
	\item[(ii)] if $\alpha_{1}>d_{1}=1$, for all $0<a\leq a_{0}$ with $a_0>0$ sufficiently small, and all $a^{\alpha_{2}}\leq s \leq 1$ we have
		\begin{equation*}
			K_{4}a^{\rho''}\leq\Exp[T(a,s)]\leq K_{5}a^{\rho'},
		\end{equation*}
		where $\rho'=1+\alpha_{2}'(1-\frac{1}{\alpha_{1}})$ and $\rho''=1+\alpha_{2}''(1-\frac{1}{\alpha_{1}})$.
\end{itemize}
\end{theorem}

\begin{proof}
(i) Assume $\alpha_{1}\leq d_{1}$ and let $\alpha_{1}'<\alpha_{1}$ be fixed. Especially, we have $d_1/\alpha_1'-1>0$. For $0<t\leq 1$ write $t=mc^{-i}$ with $m\in [1,c)$ and $i \in \nat_{0}$, then by Lemma \ref{specbound} we have
\begin{equation}\label{lbX1}
	\|X^{(1)}(t)\|\stackrel{\rm d}{=} \|c^{-iE_1}X^{(1)}(m)\|\geq \|X^{(1)}(m)\|/\|c^{iE_{1}}\|\geq K^{-1}c^{-i/{\alpha_{1}'}}\|X^{(1)}(c^it)\|.
\end{equation}
For $0<a\leq1$ choose $i_{0},i_{1}\in \nat_{0}$ such that $c^{-(i_{0}+1)} <a\leq c^{-i_{0}}$ and $c^{-(i_{1}+1)} <c^{-i_{0}\alpha_{1}'}\leq c^{-i_{1}}$. Since $X^{(1)}$ is a $(c^{E_{1}},c)$-operator semistable L\'evy process in $\rr^{d_1}\cong V_1$, the spectral component $X^{(1)}(m)$ has a bounded and continuous density $g_{m}(x_1)$ for any $m\in[1,c)$ and by Lemma \ref{densbound} there exists
\begin{equation}\label{K6}
K_6=\sup_{m\in [1,c)}\sup_{ x_1\in\rr^{d_1}}|g_{m}(x_1)|<\infty.
\end{equation}
Alltogether we observe using \eqref{lbX1}
\begin{align*}
\Exp[T(a,s)] & \leq \int_{0}^{1}P\left(\|X^{(1)}(t)\|<a\right)\,d t\leq\sum_{i=1}^{\infty}\int_{c^{-i}}^{c^{-i+1}}P\left(\|X^{(1)}(t)\|\leq c^{-i_{0}}\right)\,d t\\
	& \leq \sum_{i=1}^{\infty}\int_{1}^{c}c^{-i}P\left(\|X^{(1)}(m)\|\leq K\,c^{i/\alpha_1'-i_{0}}\right)\,dm\\
	&\leq\sum_{i=1}^{i_{1}+1}c^{-i}\int_{1}^{c}\int_{\rr^{d_1}} 1_{\{\|x_1\|\leq K\,c^{i/\alpha_1'-i_{0}}\}}g_{m}(x_1)\,dx_1\,dm+\sum_{i=i_{1}+2}^{\infty}\int_{1}^{c}c^{-i}\,d m \\
	& \leq\sum_{i=1}^{i_{1}+1}c^{-i}(c-1)(2K\,c^{i/\alpha_1'-i_{0}})^{d_1}K_6+\sum_{i=i_{1}+2}^{\infty}(c-1)c^{-i}\\
	& \leq K\,c^{-i_0d_1}\frac{\left(c^{d_1/\alpha_1'-1}\right)^{i_1+2}-1}{c^{d_1/\alpha_1'-1}-1}+c^{-(i_1+1)}\\
	& \leq K\,c^{-i_0d_1}(c^{-i_1})^{1-d_1/\alpha_1'}+c^{-i_0\alpha_1'}\leq K_3a^{\alpha_1'}.
\end{align*}
which gives the upper bound in part (i) for all $0<s\leq1$. To prove the lower bound, choose $\alpha_{j}''>0$ for $1\leq j\leq p$ such that $\alpha_{j}''>\alpha_{j}>\alpha_{j+1}''$. For $0<a\leq1$ and $a^{\alpha_1}\leq s\leq1$ choose $i_{0},i_{1},i_2\in \nat_{0}$ such that $c^{-i_{0}} <a\leq c^{-i_{0}+1}$, $c^{-i_1}<s\leq c^{-i_{1}+1}$ and $c^{-(i_{2}+1)} <\left(c^{-i_{0}}\delta\right)^{\alpha_{1}''}\leq c^{-i_{2}}$, where $0<\delta\leq1$ will be chosen later. Note that 
$$c^{-i_{1}+1}\geq s\geq a^{\alpha_1}>a^{\alpha_1''}>c^{-i_0\alpha_1''}>\left(c^{-i_{0}}\delta\right)^{\alpha_{1}''}>c^{-(i_{2}+1)}$$
and hence $i_1-1\leq i_2+1$. Similar to \eqref{lbX1}, by Lemma \ref{specbound} we have
\begin{equation}\begin{split}\label{ubXj}
	\|X^{(j)}(t)\| & \stackrel{\rm d}{=} \|c^{-iE_j}X^{(j)}(m)\|\leq \|c^{-iE_j}\|\,\|X^{(1)}(m)\|\\
	& \leq K\,c^{-i/{\alpha_{j}''}}\|X^{(j)}(c^it)\|\leq K\,c^{-i/{\alpha_{1}''}}\|X^{(j)}(c^it)\|
\end{split}\end{equation}
for all $j=1,\ldots,p$.  Alltogether we observe, using \eqref{ubXj}
\begin{align*}
\Exp[T(a,s)] & \geq \int_{0}^{s}P\left(\|X^{(j)}(t)\|<\frac{a}{\sqrt{p}},\;1\leq j\leq p\right)\,d t\\
	& \geq \int_{0}^{c^{-i_{1}}}P\left(\|X^{(j)}(t)\|\leq\frac{c^{-i_{0}}}{\sqrt{p}},\;1\leq j\leq p\right)\,d t\\
	& = \sum_{i=i_{1}-1}^{\infty}\int_{c^{-i}}^{c^{-i+1}}P\left(\|X^{(j)}(c^it)\|\leq K^{-1}\frac{c^{i/\alpha_1''-i_{0}}}{\sqrt{p}},\;1\leq j\leq p\right)\,d t\\
	& \geq\sum_{i=i_{2}+1}^{\infty}c^{-i}\int_{1}^{c}P\left(\|X^{(j)}(m)\|\leq K^{-1}\frac{c^{(i_2+1)/\alpha_1''-i_{0}}}{\sqrt{p}},\;1\leq j\leq p\right)\,dm\\
	& \geq\sum_{i=i_{2}+1}^{\infty}c^{-i}\int_{1}^{c}P\left(\|X^{(j)}(m)\|\leq \frac{K^{-1}}{\delta\sqrt{p}},\;1\leq j\leq p\right)\,dm
\end{align*}
Since $\{X^{(j)}(t)\}_{t\geq 0}$, $1\leq j\leq p$, are L\'evy processes, we can assume that they have c\`adl\`ag paths. Hence $\sup_{m\in [1,c)} \|X^{(j)}(m)\| = \sup_{m\in [1,c)\cap \mathbb{Q}} \|X^{(j)}(m)\|$, $1\leq j\leq p$, are random variables and thus
$$P\left(\sup_{m\in [1,c)} \|X^{(j)}(m)\| \leq \frac{K^{-1}}{\delta\sqrt{p}},\;1\leq j\leq p\right)\geq K_{7}>0,$$
if we choose $0<\delta\leq1$ sufficiently small. Consequently,
\begin{align*}
\Exp[T(a,s)] & \geq\sum_{i=i_{2}+1}^{\infty}c^{-i}\int_{1}^{c}K_7\,dm= K_{7} \sum_{i=i_{2}+1}^{\infty}c^{-i}(c-1)=K_7c^{-i_2}\\
& \geq K (c^{-i_0}\delta)^{\alpha_1''}=K(\delta/c)^{\alpha_1''}a^{\alpha_1''}=K_2a^{\alpha_1''} 
\end{align*}
which proves the lower bound in part (i).
	
(ii) Now assume $\alpha_{1}>d_{1}=1$ and let $\alpha_{2}'<\alpha_{2}$ be fixed. Since $(X^{(1)}, X^{(2)})$ is a $(c^{E_{1}\oplus E_{2}},c)$-semistable L\'evy process in $\mathbb{R}^{d_{1}+d_{2}}\cong V_{1}\oplus V_{2}$, the spectral component $(X^{(1)}(m), X^{(2)}(m))$ has a bounded and continuous density $g_{m}(x_1,x_2)$ for any $m\in[1,c)$ and by Lemma \ref{densbound} there exists
\begin{equation}\label{K8}
K_8=\sup_{m\in [1,c)}\sup_{(x_{1},x_{2})\in\rr^{d_1+d_2}}|g_{m}(x_1,x_2)|<\infty.
\end{equation}
We will further use the constant $K_6$ defined by \eqref{K6} in part (i). For $0<a\leq1$ choose $i_{0},i_{1}\in \nat_{0}$ such that $c^{-(i_{0}+1)} <a\leq c^{-i_{0}}$ and $c^{-(i_{1}+1)} <c^{-i_{0}\alpha_{2}'}\leq c^{-i_{1}}$. For $0<t\leq 1$ again write $t=mc^{-i}$ with $m\in [1,c)$ and $i \in \nat_{0}$, then by Lemma \ref{specbound} we have
\begin{equation}\label{lbX2}
	\|X^{(2)}(t)\|\stackrel{\rm d}{=} \|c^{-iE_2}X^{(2)}(m)\|\geq \|X^{(2)}(m)\|/\|c^{iE_{2}}\|\geq K^{-1}c^{-i/{\alpha_{2}'}}\|X^{(2)}(c^it)\|.
\end{equation}
Alltogether we observe using \eqref{lbX2}
\begin{align*}
\Exp[T(a,s)] & \leq \int_{0}^{1} P\left(|X^{(1)}(t)|<a,\|X^{(2)}(t)\|<a\right)\,d t\\
	& \leq \sum_{i=1}^{\infty}\int_{c^{-i}}^{c^{-i+1}}P\left(|X^{(1)}(c^it)|<c^{i/\alpha_1-i_{0}},\|X^{(2)}(c^it)\|<K\,c^{i/\alpha_2'-i_{0}}\right)\,d t\\
	& \leq \sum_{i=1}^{i_1+1}c^{-i}\int_{1}^{c}P\left(|X^{(1)}(m)|<c^{i/\alpha_1-i_{0}},\|X^{(2)}(m)\|<K\,c^{i/\alpha_2'-i_{0}}\right)\,dm\\
	& \phantom{\leq}+\sum_{i=i_1+2}^{\infty}c^{-i}\int_{1}^{c}P\left(|X^{(1)}(m)|<c^{i/\alpha_1-i_{0}}\right)\,dm\\
	& =: I+I\!I.
\end{align*}
Note that for part $I$ we have $\alpha_{2}' < \alpha_{1} < 2$ and $ d_{2}\geq 1$, hence $1-\frac{1}{\alpha_{1}}-\frac{d_{2}}{\alpha_{2}'}<0$ and it follows that
\begin{align*}
I & \leq\sum_{i=1}^{i_{1}+1}c^{-i}(c-1)K_82c^{i/\alpha_1-i_{0}}\left(2c^{i/\alpha_2'-i_{0}}\right)^{d_{2}}\\
	&\leq Kc^{-i_{0}(d_{2}+1)}\sum_{i=1}^{i_{1}+1}\left(c^{-i}\right)^{1-\frac{1}{\alpha_{1}}-\frac{d_{2}}{\alpha_{2}'}}
	= Kc^{-i_{0}(d_{2}+1)}\left[\frac{\left(c^{-(i_{1}+2)}\right)^{1-\frac{1}{\alpha_{1}}-\frac{d_{2}}{\alpha_{2}'}}-1}{c^{\frac{1}{\alpha_{1}}+\frac{d_{2}}{\alpha_{2}'}-1}-1}\right]\\
	&\leq Kc^{-i_{0}(d_{2}+1)}\left(c^{-i_{1}}\right)^{1-\frac{1}{\alpha_{1}}-\frac{d_{2}}{\alpha_{2}'}}\left(c^{-2}\right)^{1-\frac{1}{\alpha_{1}}-\frac{d_{2}}{\alpha_{2}'}}
	\leq Kc^{-i_{0}(d_{2}+1)}\left(c^{-i_{0}\alpha_{2}'}\right)^{1-\frac{1}{\alpha_{1}}-\frac{d_{2}}{\alpha_{2}'}}\\
	&= K\left(c^{-i_{0}}\right)^{1+\alpha_{2}'(1-\frac{1}{\alpha_{1}})} = Kc^{-i_{0}\rho'}
	= Kc^{\rho'}a^{\rho'}=K_{51}a^{\rho'}
\end{align*}
Further note that for part $I\!I$ we have $\alpha_{1}>1$, hence $1-\frac{1}{\alpha_{1}}>0$ and
\begin{align*}
I\!I &\leq \sum_{i=i_{1}+2}^{\infty}c^{-i}(c-1)K_62c^{i/\alpha_1-i_{0}}
	= Kc^{-i_{0}}\sum_{i=i_{1}+2}^{\infty}\left(c^{-i}\right)^{1-\frac{1}{\alpha_{1}}}(c-1)\\
	&= Kc^{-i_{0}}\left(c^{-(i_{1}+2)}\right)^{1-\frac{1}{\alpha_{1}}}
	\leq Kc^{-i_{0}}\left(c^{-i_{0}\alpha_{2}'}\right)^{1-\frac{1}{\alpha_{1}}}\\
	&= K\left(c^{-i_{0}}\right)^{1+\alpha_{2}'(1-\frac{1}{\alpha_{1}})} = Kc^{-i_{0}\rho'}\leq K_{52}a^{\rho'}
\end{align*}
Putting things together, we get the upper bound $\Exp[T(a,s)] \leq K_{51}a^{\rho'}+ K_{52}a^{\rho'} = K_{5}a^{\rho'}$ in part (ii) for all $0\leq s\leq 1$. To prove the lower bound, we choose $i_0,i_1$ as in the proof of the lower bound in part (i), i.e. $c^{-i_0}<a\leq c^{-i_0+1}$ and $c^{-i_1}<s\leq c^{-i_1+1}$. Note that, since $d_1=1$, for $j=1$ in \eqref{ubXj} we can choose $K=1$ and $\alpha_1''=\alpha_1$. Hence, similar to the above, we get
\begin{equation}\label{stlb}
\Exp[T(a,s)]\geq\sum_{i=i_{1}-1}^{\infty}c^{-i}\int_{1}^{c}P\left(\begin{array}{c} |X^{(1)}(m)|<\frac{c^{i/\alpha_1-i_{0}}}{\sqrt{p}}\text{ and}\\ \|X^{(j)}(m)\|\leq K^{-1}\frac{c^{i/\alpha_j''-i_{0}}}{\sqrt{p}}, 2\leq j\leq p\end{array}\right)dm.
\end{equation}
By Lemma \ref{positive} choose $K_{10}>0$, $r>0$ and uniformly bounded Borel sets $J_m\subseteq\rr^{d-1}$ with Lebesgue measure $0<K_9\leq\lambda^{d-1}(J_m)<\infty$ for every $m\in[1,c)$ such that the bounded continuous density $g_m(x_1,\ldots,x_p)$ of $X(m)=X^{(1)}(m)+\cdots+X^{(p)}(m)$ fulfills
$$g_m(x_1,\ldots,x_p)\geq K_{10}>0\quad\text{ for all }(x_1,\ldots,x_p)\in[-r,r]\times J_{m}$$
and for every $m\in[1,c)$. Since $\{J_m\}_{m\in[1,c)}$ is uniformly bounded by Lemma \ref{positive}, we are able to choose $0<\delta\leq c^{-1}<1$ such that
$$\bigcup_{m\in[1,c)}J_{m}\subseteq\left\{\|x_j\|\leq\frac{K^{-1}c^{-\alpha_1/\alpha_p}}{\delta\sqrt{p}},\;2\leq j\leq p\right\},$$
where $K$ is the constant from \eqref{stlb}. Let $\eta=c^{2/\alpha_p}/(r\sqrt{p})$. Since $\alpha_{1}>\alpha_{2}''$, there exists a constant $a_{0}\in(0,1]$ such that $\left(\eta a\right)^{\alpha_{1}}<\left(\delta a\right)^{\alpha_{2}''}$ for all $0<a\leq a_{0}$. Now chose $i_{2}, i_{3}\in \mathbb{N}_{0}$ such that $c^{-i_{2}} <\left(\delta c^{-i_{0}+1}\right)^{\alpha_{2}''}\leq c^{-i_{2}+1}$ and $c^{-i_{3}} <\left(\eta c^{-i_{0}}\right)^{\alpha_{1}}\leq c^{-i_{3}+1}$. Note that
$$c^{-i_3}<\left(\eta c^{-i_{0}}\right)^{\alpha_{1}}<\left(\eta a\right)^{\alpha_{1}}<\left(\delta a\right)^{\alpha_{2}''}\leq\left(\delta c^{-i_0+1}\right)^{\alpha_{2}''}\leq c^{-i_2+1}$$
and
$$c^{-(i_1-1)}\geq s\geq a^{\alpha_2}\geq a^{\alpha_2''}>\left(c^{-i_0}\right)^{\alpha_2''}\geq\left(\delta c^{-i_0+1}\right)^{\alpha_2''}>c^{-i_2},$$
hence $i_3\geq i_2-1$ and $i_1-1\leq i_2$. We further have for all $i=i_2,\ldots,i_3+1$ and every $j=2,\ldots,p$
\begin{equation}\label{up}
\frac{c^{i/\alpha_1-i_{0}}}{\sqrt{p}}\leq \frac{c^{(i_3+1)/\alpha_1-i_{0}}}{\sqrt{p}}\leq\frac{c^{2/\alpha_1}(\eta c^{-i_{0}})^{-1}c^{-i_{0}}}{\sqrt{p}}=\frac{c^{2/\alpha_1}}{\eta\sqrt{p}}=r
\end{equation}
and
\begin{equation}\begin{split}\label{down}
\frac{c^{i/\alpha_j''-i_{0}}}{\sqrt{p}} & \geq \frac{c^{i_2/\alpha_j''-i_{0}}}{\sqrt{p}}\geq \frac{(\delta c^{-i_0+1})^{-\alpha_2''/\alpha_j''}c^{-i_0}}{\sqrt{p}}\\
& =\frac{(\delta^{-1}c^{i_0-1})^{\alpha_2''/\alpha_j''}c^{-i_0}}{\sqrt{p}}\geq\frac{c^{-\alpha_2''/\alpha_j''}}{\delta\sqrt{p}}\geq\frac{c^{-\alpha_1/\alpha_p}}{\delta\sqrt{p}}.
\end{split}\end{equation}
Let $I_m=(-\frac{c^{i/\alpha_1-i_{0}}}{\sqrt{p}},\frac{c^{i/\alpha_1-i_{0}}}{\sqrt{p}})\times J_m$ then in view of \eqref{stlb}, we get using \eqref{up} and \eqref{down}
\begin{align*}
\Exp[T(a,s)] & \geq \sum_{i=i_{2}}^{i_3+1}c^{-i}\int_{1}^{c}P\left(\begin{array}{c} |X^{(1)}(m)|<\frac{c^{i/\alpha_1-i_{0}}}{\sqrt{p}}\text{ and}\\ \|X^{(j)}(m)\|\leq K^{-1}\frac{c^{i/\alpha_j''-i_{0}}}{\sqrt{p}}, 2\leq j\leq p\end{array}\right)dm\\
	&\geq \sum_{i=i_{2}}^{i_{3}+1}c^{-i}\int_{1}^{c}\int_{I_m} g_{m}(x)\,dx\,dm
	\geq \sum_{i=i_{2}}^{i_{3}+1}c^{-i}(c-1)\,2\,\frac{c^{i/\alpha_1-i_{0}}}{\sqrt{p}}\,K_9\,K_{10}\\
	&= Kc^{-i_{0}}\sum_{i=i_{2}}^{i_{3}+1}\left(c^{-i}\right)^{1-\frac{1}{\alpha_{1}}}
	= Kc^{-i_{0}}\left(\frac{1-\left(c^{-(i_{3}+2)}\right)^{1-\frac{1}{\alpha_{1}}}}{1-c^{\frac{1}{\alpha_{1}}-1}}- \frac{1-\left(c^{-i_{2}}\right)^{1-\frac{1}{\alpha_{1}}}}{1-c^{\frac{1}{\alpha_{1}}-1}}\right)\\
	& = Kc^{-i_{0}}\left(\left(c^{-i_{2}}\right)^{1-\frac{1}{\alpha_{1}}}-\left(c^{-(i_{3}+2)}\right)^{1-\frac{1}{\alpha_{1}}}\right)\\	
	& \geq K_{41}\left(c^{-i_{0}}\right)^{\rho''} - K_{42}\left(c^{-i_{0}}\right)^{\alpha_{1}}.
\end{align*}
Since $\rho''=1+\alpha_{2}''(1-\frac{1}{\alpha_{1}}) < 1+\alpha_{1}(1-\frac{1}{\alpha_{1}})=\alpha_{1}$ we have
$\left(c^{-i_{0}}\right)^{\alpha_{1}-\rho''}\to0$ if $a\to0$, i.e. $i_{0}\to\infty$. Hence we can further choose $a_{0}$ sufficiently small, such that
$$\Exp[T(a,s)]\geq\frac{K_{41}}{2}\,\left(c^{-i_{0}}\right)^{\rho''}\geq K_{4}a^{\rho''}$$
for all $0<a\leq a_{0}$, which proves the lower bound in part (ii) and concludes the proof.
\end{proof}

\begin{remark}\label{rem1dim}
In fact we have proven a little bit more than stated in Theorem \ref{sojournbounds}. Part (i) is also valid in case $d=1$ for a $(c^{1/\alpha},c)$-semistable L\'evy process in $\rr$ with $\alpha_1=\alpha$ and $d_1=1$. Our proof also shows that the upper bounds in part (i) and (ii) are valid for all $0\leq s\leq 1$, but this is also a direct consequence from the definition of a sojourn time.
\end{remark}	

\section{Main Results}

Recall the spectral decomposition of Section 2.2 for the constants $\alpha_1,\alpha_2$ and $d_1$ appearing in the following results.

\begin{theorem}\label{main}
Let $X=\{X(t)\}_{t\geq0}$ be an operator semistable L\'evy process in  $\rd$ with $d\geq2$. Then for any Borel set $B\subseteq\rr_+$ we have almost surely 
\begin{align*}
	\dim_{\rm H}X(B)=
	\begin{cases} 
	\alpha_1\dim_{\rm H}B & \text{if } \alpha_1\dim_{\rm H}B \leq d_1, \\
	1+\alpha_2\left(\dim_{\rm H}B-\frac1{\alpha_1}\right) & \text{if } \alpha_1\dim_{\rm H}B > d_1.
	\end{cases}
\end{align*}
\end{theorem}

As a direct consequence, for $B=[0,1]$ with $\dim_{\rm H}B=1$ the Hausdorff dimension of the range of $X$ is determined as follows.

\begin{cor}
Let $X=\{X(t)\}_{t\geq0}$ be an operator semistable L\'evy process in  $\rd$ with $d\geq2$. Then we have almost surely 
\begin{align*}
	\dim_{\rm H}X([0,1])=
	\begin{cases} 
	\alpha_1 & \text{if } \alpha_1\leq d_1, \\
	1+\alpha_2\left(1-\frac1{\alpha_1}\right) & \text{otherwise.}
	\end{cases}
\end{align*}
\end{cor}

The lower cases in the above dimension formulas are only meaningful if $d\geq2$. For a one-dimensional semistable L\'evy process the Hausdorff dimension is determined as follows. 

\begin{theorem}\label{dim1}
Let $X=\{X(t)\}_{t\geq0}$ be a $(c^{1/\alpha},c)$-semistable L\'evy process in  $\rr$. Then for any Borel set $B\subseteq\rr_+$ we have almost surely 
$$\dim_{\rm H}X(B)=\min(\alpha\dim_{\rm H}B,1).$$
In particular, for $B=[0,1]$ we obtain for the range $\dim_{\rm H}X([0,1])=\min(\alpha,1)$ a.s.
\end{theorem}

For the proof of Theorem \ref{main} we follow standard techniques of determining upper and lower bounds for $\dim_{\rm H}X(B)$ as described on page 289 of \cite{Xiao}. Similar arguments can be found in Xiao and Lin \cite{XL} for multivariate selfsimilar processes with independent components.

\subsection{Upper bounds}

To obtain upper bounds for $\dim_{\rm H}X(B)$ we choose a suitable sequence of coverings of $X(B)$ and show that its corresponding $\gamma$-dimensional Hausdorff measure has finite expectation, which leads to $\dim_{\rm H}X(B)\leq\gamma$ almost surely. This method goes back to Pruitt and Taylor \cite{PT} and Hendricks \cite{Hen}.
\begin{lemma}\label{upperbound}
Let $X=\{X(t)\}_{t\geq0}$ be an operator semistable L\'evy process in  $\rd$ with $d\geq2$. Then for any Borel set $B \subseteq \rr_+$ we have almost surely 
\begin{align*}
	\dim_{\rm H}X(B)\leq
	\begin{cases} 
	\alpha_1\dim_{\rm H}B & \text{if } \alpha_1\dim_{\rm H}B \leq d_1, \\
	1+\alpha_2\left(\dim_{\rm H}B-\frac1{\alpha_1}\right) & \text{if } \alpha_1\dim_{\rm H}B > d_1.
	\end{cases}
\end{align*}
\end{lemma}

\begin{proof}
(i) Assume $\alpha_1\dim_{\rm H}B \leq d_1$ and $\alpha_{1}\leq d_{1}$. For $\gamma>\dim_{\rm H}B$ choose $\alpha_{1}''>\alpha_{1}$ such that $\gamma'=1-\frac{\alpha_{1}''}{\alpha_{1}}+\gamma>\dim_{\rm H}B$. Then, by definition of the Hausdorff dimension, for any $\varepsilon\in(0,1]$ there exists a sequence $\{I_{i}\}_{i\in\mathbb{N}}$ of intervals in $\rr_{+}$ of length $|I_{i}|<\varepsilon$ such that
$$B\subseteq \bigcup_{i=1}^{\infty}I_{i}\quad\text{ and }\quad\sum_{i=1}^{\infty}|I_{i}|^{\gamma'}<1.$$
Let $s_{i}=|I_{i}|$ und $b_{i}:=|I_{i}|^{\frac{1}{\alpha_{1}}}$ then $(b_i/3)^{\alpha_1}<s_i$. By Lemma \ref{PT} and Theorem \ref{sojournbounds} it follows that $X(I_{i})$ can be covered by $M_{i}$ cubes $C_{ij}\in \Lambda(b_{i})$ of side $b_{i}$ such that for every $i\in\nat$ we have
$$\Exp[M_{i}]\leq 2K_{1}s_{i}\left(\Exp\left(T\left(\tfrac{b_{i}}{3},s_{i}\right)\right]\right)^{-1}\leq 2K_{1}s_{i}K_{2}^{-1}\left(\tfrac{b_{i}}{3}\right)^{-\alpha_{1}''}= K \,s_{i}b_{i}^{-\alpha_{1}''} = K \,|I_{i}|^{1-\frac{\alpha_{1}''}{\alpha_{1}}}.$$
Note that $X(B)\subseteq\bigcup_{i=1}^{\infty} \bigcup_{j=1}^{M_{i}}C_{ij}$, where $b_{i}\sqrt{d}$ is the diameter of $C_{ij}$. Hence $\{C_{ij}\}$ is a $(\varepsilon^{1/\alpha_{1}}\sqrt{d})$-covering of $X(B)$. By monotone convergence we have
\begin{align*}
\Exp\left[\sum_{i=1}^{\infty}M_{i}b_{i}^{\alpha_{1}\gamma}\right] & = \sum_{i=1}^{\infty}\Exp\left[M_{i}b_{i}^{\alpha_{1}\gamma}\right]\leq \sum_{i=1}^{\infty} K \,|I_{i}|^{1-\frac{\alpha_{1}''}{\alpha_{1}}}\, |I_{i}|^{\gamma}= K\sum_{i=1}^{\infty}|I_{i}|^{\gamma'} \leq K.
\end{align*}
Letting $\varepsilon\to0$, i.e $b_{i}\to0$, by Fatou's lemma we get
\begin{align*}
\Exp\left[\mathcal{H}^{\alpha_{1}\gamma}(X(B))\right] & \leq \Exp\left[\liminf_{\varepsilon\to0} \sum_{i=1}^{\infty}\sum_{j=1}^{M_{i}}\left(b_{i}\sqrt{d}\right)^{\alpha_{1}\gamma}\right]\\
& \leq \liminf_{\varepsilon\to 0} \sqrt{d}^{\alpha_{1}\gamma}\Exp\left[\sum_{i=1}^{\infty}M_{i}b_{i}^{\alpha_{1}\gamma}\right]\leq\sqrt{d}^{\alpha_{1}\gamma} K < \infty,
\end{align*}
which shows that $\dim_{\rm H}X(B)\leq \alpha_{1}\gamma$ almost surely. Since $\gamma>\dim_{\rm H}B$ is arbitrary, we get $\dim_{\rm H}X(B) \leq \alpha_{1}\dim_{\rm H}B$ a.s.	

(ii) Assume $\alpha_1\dim_{\rm H}B\leq d_1$ and  $\alpha_{1}>d_{1}$. To be able to argue the same way as in part (i), we have to show that the same lower bound $\Exp[T(a,s)]\geq K\,a^{\alpha_{1}''}$ holds for the expected sojourn time also in case $\alpha_{1}>d_{1}$. In fact, by Theorem \ref{sojournbounds} (ii) we have $\Exp[T(a,s)]\geq K\,a^{\rho''}$, where $\rho''=1+\alpha_{2}''(1-\frac{1}{\alpha_{1}})$ and $0<\alpha_{2}<\alpha_{2}''<\alpha_{1}<\alpha_{1}''$.  Hence
$$\rho''=1+\alpha_{2}''(1-\tfrac{1}{\alpha_{1}})\leq 1+\alpha_{1}(1-\tfrac{1}{\alpha_{1}})=\alpha_{1}<\alpha_{1}''$$
so that for all $0<a\leq 1$ and $a^{\alpha_1}\leq s\leq1$ we get the desired lower bound. Now, as in part (i) the same conclusion 
$\dim_{\rm H}X(B)\leq \alpha_{1}\dim_{\rm H}B$ holds a.s.
	
(iii) Assume $\alpha_1\dim_{\rm H}B>d_1$. Since $\dim_{\rm H}B\leq 1$ it follows that $\alpha_{1}>d_{1}=1$. For $\gamma>\dim_{\rm H}B$ choose $\alpha_{2}''>\alpha_{2}$ such that $\gamma'=1-\frac{\alpha_{2}''}{\alpha_{2}}+\frac{\alpha_{2}''}{\alpha_{2}}\gamma>\dim_{\rm H}B$. For $\varepsilon\in(0,1]$ let $\{I_{i}\}_{i\in\mathbb{N}}$ be the same sequence of intervals as in part (i). 
Let $s_{i}:=|I_{i}|$ und $b_{i}:=|I_{i}|^{\frac{1}{\alpha_{2}}}$ then $(b_i/3)^{\alpha_2}<s_i$. Again, by Lemma \ref{PT} and Theorem \ref{sojournbounds} it follows that $X(I_{i})$ can be covered by $M_{i}$ cubes $C_{ij}\in \Lambda(b_{i})$ of side $b_{i}$ such that for every $i\in\nat$ we have
$$\Exp[M_{i}]\leq 2K_{1}s_{i}\left(\Exp\left(T\left(\tfrac{b_{i}}{3},s_{i}\right)\right]\right)^{-1}\leq 2K_{1}s_{i}K_{4}^{-1}\left(\tfrac{b_{i}}{3}\right)^{-\rho''}= K \,s_{i}b_{i}^{-\rho''} = K \,|I_{i}|^{1-\frac{\rho''}{\alpha_2}},$$
where $\rho''=1+\alpha_{2}''(1-\frac{1}{\alpha_{1}})$. By monotone convergence we have
\begin{align*}
\Exp\left[\sum_{i=1}^{\infty}M_{i}b_{i}^{1+\alpha_{2}''(\gamma-\frac1{\alpha_1})}\right] & \leq \sum_{i=1}^{\infty} K \,|I_{i}|^{1-\frac{\rho''}{\alpha_2}}\, |I_{i}|^{\frac1{\alpha_2}+\frac{\alpha_{2}''}{\alpha_2}(\gamma-\frac1{\alpha_1})}= K\sum_{i=1}^{\infty}|I_{i}|^{\gamma'} \leq K.
\end{align*}
Since $\gamma>\dim_{\rm H}B$ and $\alpha_2''>\alpha_2$ are arbitrary, with the same arguments as in part (i) we get $\dim_{\rm H}X(B) \leq 1+\alpha_2(\dim_{\rm H}B-\frac1{\alpha_1})$ a.s.	
\end{proof}

\subsection{Lower bounds}

In order to show $\dim_{\rm H}X(B)\geq\gamma$ almost surely, we use standard capacity arguments. By Frostman's lemma we choose a suitable probability measure on $B$ with finite energy and show that a corresponding random measure on $X(B)$ has finite expected $\gamma$-energy. The relationship between the Hausdorff and the capacitary dimension by Frostman's theorem then gives the desired lower bound. 
\begin{lemma}\label{lowerbound}
Let $X=\{X(t)\}_{t\geq0}$ be an operator semistable L\'evy process in  $\rd$ with $d\geq2$. Then for any Borel set $B\subseteq\rr_+$ we have almost surely 
\begin{align*}
	\dim_{\rm H}X(B)\geq
	\begin{cases} 
	\alpha_1\dim_{\rm H}B & \text{if } \alpha_1\dim_{\rm H}B \leq d_1, \\
	1+\alpha_2\left(\dim_{\rm H}B-\frac1{\alpha_1}\right) & \text{if } \alpha_1\dim_{\rm H}B > d_1.
	\end{cases}
\end{align*}
\end{lemma}

\begin{proof}
First assume $0<\alpha_1\dim_{\rm H}B\leq d_1$. In case $\dim_{\rm H}B=0$ there is nothing to prove. For $0<\gamma<\alpha_{1}\dim_{\rm H}B$ choose $0<\alpha_{1}'<\alpha_{1}$ such that $\gamma<\alpha_{1}'\dim_{\rm H}B$. By Frostman's lemma \cite{Kah,Mat} there exists a probability measure $\sigma$ on $B$ such that
\begin{equation}\label{sigma}
	\int_{B}\int_{B}\frac{\sigma(ds)\,\sigma(dt)}{|s-t|^{\gamma/\alpha_{1}'}}<\infty.
\end{equation}
In order to prove $\dim_{\rm H}X(B)\geq\gamma$ almost surely, by Frostman's theorem \cite{Kah,Mat} it suffices to show that
\begin{equation}\label{Xsigma}
	\int_{B}\int_{B}\Exp\left[\|X(s)-X(t)\|^{-\gamma}\right]\,\sigma(ds)\,\sigma(dt)<\infty.
\end{equation}
Let $K_{11}=\sup_{m\in[1,c)} E(\|X^{(1)}(m)\|^{-\gamma})<\infty$ by Lemma \ref{negmom}, since $\gamma<\alpha_{1}\dim_{\rm H}B\leq d_{1}$. In order to verify \eqref{Xsigma} we split the domain of integration into two parts

(i) Assume $|s-t|\leq 1$, then $|s-t|=mc^{-i}$ with $m\in[1,c)$ and $i\in\mathbb{N}_{0}$. By Lemma \ref{specbound} we get
\begin{align*}
	\Exp\left[\|X(s)-X(t)\|^{-\gamma}\right] & \leq \Exp\left[\|X^{(1)}(mc^{-i})\|^{-\gamma}\right]= \Exp\left[\|c^{-iE_{1}}X^{(1)}(m)\|^{-\gamma}\right]\\
	& \leq  \|c^{iE_{1}}\|^{\gamma}\Exp\left[\|X^{(1)}(m)\|^{-\gamma}\right]\leq K\,c^{\gamma i/\alpha_{1}'} K_{11}\\
	&= Km^{\frac{\gamma}{\alpha_{1}'}}\cdot \left(mc^{-i}\right)^{-\frac{\gamma}{\alpha_{1}'}}\leq =K_{12}|s-t|^{-\frac{\gamma}{\alpha_{1}'}}.
	\end{align*}
(ii) Now assume $|s-t|\geq 1$ and choose $\alpha_{1}''>\alpha_{1}$. Write $|s-t|=mc^{i}$ with $m\in[1,c)$ and $i\in\mathbb{N}_{0}$. Then, using again Lemma \ref{specbound} we get as above
\begin{align*}
	\Exp\left[\|X(s)-X(t)\|^{-\gamma}\right] & = \|c^{-iE_{1}}\|^{\gamma} \Exp\left[\|X^{(1)}(m)\|^{-\gamma}\right]\leq K\,c^{-\gamma i/\alpha_{1}''} K_{11}\leq K\,K_{11} = K_{13}.
\end{align*}

Combining part (i) and (ii) in \eqref{Xsigma}, by \eqref{sigma} we get the desired upper bound in case $\alpha_1\dim_{\rm H}B\leq d_1$.

Now assume $\alpha_1\dim_{\rm H}B>d_1$, then $\alpha_{1}>d_{1}=1$ and hence $\dim_{\rm H}B>\frac{1}{\alpha_{1}}$. Choose $1<\gamma<1+\alpha_{2}(\dim_{\rm H}B-\frac{1}{\alpha_{1}})$, then since $\rho=\frac{\gamma}{\alpha_{2}}-(\frac{1}{\alpha_{2}}-\frac{1}{\alpha_{1}})<\dim_{\rm H}B$ we can choose $0<\alpha_{2}'<\alpha_{2}$ such that $\rho'=\frac{\gamma}{\alpha_{2}'}-(\frac{1}{\alpha_{2}'}-\frac{1}{\alpha_{1}})<\dim_{\rm H}B$. By Frostman's lemma there exists again a probability measure $\sigma$ on $E$ such that
\begin{equation}\label{sigma1}
	\int_{B}\int_{B}\frac{\sigma(ds)\,\sigma(dt)}{|s-t|^{\rho'}}<\infty.
\end{equation}
Again, in order to show \eqref{Xsigma} we split the domain of integration into two parts.

(i) Assume $|s-t|=mc^{-i}\leq 1$ with $m\in[1,c)$ and $i\in\mathbb{N}_{0}$. By Lemma \ref{specbound} we get
\begin{align*}
& \Exp\left[\|X(s)-X(t)\|^{-\gamma}\right]=\Exp\left[\|c^{-iE}X(m)\|^{-\gamma}\right]\\
	& \leq \Exp\left[\left(c^{-i\frac{2}{\alpha_{1}}}|X^{(1)}(m)|^{2} + \|X^{(2)(m)}\|^{2}/ \|c^{iE_{2}}\|^{2} \right)^{-\frac{\gamma}{2}}\right]\\
	&\leq K\int_{\rr^{1+d_2}} \frac{1}{c^{-i\frac{\gamma}{\alpha_{1}}}\left|x_{1}\right|^{\gamma} + c^{-i\frac{\gamma}{\alpha_{2}'}}\left\|x_{2}\right\|^{\gamma} } \,g_{m}(x_{1},x_{2}) \,dx_{1}\,dx_{2}\\
	&= K\int_{\rr^{1+d_2}} \frac{1}{m^{-\frac{\gamma}{\alpha_{1}}}\left(mc^{-i}\right)^{\frac{\gamma}{\alpha_{1}}}\left|x_{1}\right|^{\gamma} + m^{-\frac{\gamma}{\alpha_{2}'}}\left(mc^{-i}\right)^{\frac{\gamma}{\alpha_{2}'}}\left\|x_{2}\right\|^{\gamma}} \,g_{m}(x_{1},x_{2}) \,dx_{1}\,dx_{2}\\
	& \leq K\int_{\rr^{1+d_2}} \frac{1}{c^{-\frac{\gamma}{\alpha_{1}}}\left|s-t\right|^{\frac{\gamma}{\alpha_{1}}}\left|x_{1}\right|^{\gamma} + c^{-\frac{\gamma}{\alpha_{2}'}}\left|s-t\right|^{\frac{\gamma}{\alpha_{2}'}}\left\|x_{2}\right\|^{\gamma}} \,g_{m}(x_{1},x_{2}) \,dx_{1}\,dx_{2}\\
	& \leq K \int_{\rr^{1+d_2}} \frac{1}{\left|s-t\right|^{\frac{\gamma}{\alpha_{1}}}\left|x_{1}\right|^{\gamma} + \left|s-t\right|^{\frac{\gamma}{\alpha_{2}'}}\left\|x_{2}\right\|^{\gamma}}\, g_{m}(x_{1},x_{2}) \,dx_{1}\,dx_{2}\\
	&= K\left|s-t\right|^{-\frac{\gamma}{\alpha_{1}}}\int_{\rr^{1+d_2}} \frac{1}{\left|x_{1}\right|^{\gamma} + \left|s-t\right|^{\gamma(\frac{1}{\alpha_{2}'}-\frac{1}{\alpha_{1}})}\left\|x_{2}\right\|^{\gamma} } \,g_{m}(x_{1},x_{2}) \,dx_{1}\,dx_{2},
	\end{align*}
where $g_{m}(x_{1},x_{2})$ denotes a bounded continuous density of $(X^{(1)}(m),X^{(2)}(m))$ in $\rr^{1+d_2}\cong V_1\oplus V_2$. We will use integration by parts to derive an upper bound for the above integral $I$. Let
	\begin{align*}
		F_{m}(r_{1}, r_{2})=P\left(|X^{(1)}(m)|\leq r_{1}, \|X^{(2)}(m)\|\leq r_{2}\right).
	\end{align*}
which by transformation into spherical coordinates reads as 
	\begin{align*}
		F_{m}(r_{1}, r_{2})&=\int_{|x_{1}|\leq r_{1}}\int_{\|x_{2}\|\leq r_{2}} g_{m}(x_{1},x_{2}) \,dx_{1}\,dx_{2}\\
		&= \int_{-r_{1}}^{r_{1}} \int_{0}^{r_{2}} \int_{S_{d_{2}-1}} \tilde{g}_{m}(\rho_{1},\rho_{2}\theta)\rho_{2}^{d_{2}-1} \mu(\,d\theta)\,d\rho_{2}\,d\rho_{1},
	\end{align*}
	where $\tilde{g}_{m}(\rho_{1},\rho_{2}\theta)$ is a bounded continuous function in  $(\rho_{1},\rho_{2},\theta)\in\mathbb{R}\times\mathbb{R}_{+}\times S_{d_{2}-1}$ and $\mu$ is the surface measure on the unit sphere $S_{d_{2}-1}$ in $\mathbb{R}^{d_{2}}$. Note that by \eqref{K8} we have
	\begin{equation}\label{K8neu}
		\sup_{m\in[1,c)}\sup_{(\rho_{1},\rho_{2},\theta)\in\mathbb{R}\times\mathbb{R}_{+}\times S_{d_{2}-1}}\tilde{g}_{m}(\rho_{1},\rho_{2}\theta)=K_8<\infty.
	\end{equation}
For simplicity let $z=|s-t|^{\frac{1}{\alpha_{2}'}-\frac{1}{\alpha_{1}}}$. By Fubini's theorem and integration by parts with respect to $dr_{1}$ we get for the above integral $I$
	\begin{align*}
		I&=\int_{0}^{\infty}\int_{0}^{\infty} \frac{1}{r_{1}^{\gamma}+z^{\gamma}r_{2}^{\gamma}} \,F_{m}(\,dr_{1},\,dr_{2})\\
		&= \int_{0}^{\infty}\int_{0}^{\infty} \frac{1}{r_{1}^{\gamma}+z^{\gamma}r_{2}^{\gamma}} \int_{S_{d_{2}-1}} \tilde{g}_{m}(r_{1},r_{2}\theta) \,r_{2}^{d_{2}-1} \mu(\,d\theta)\,dr_{1}\,dr_{2}\\
		&=  0+\int_{0}^{\infty} \int_{0}^{\infty} \left[\frac{\gamma r_{1}^{\gamma-1}}{(r_{1}^{\gamma}+z^{\gamma}r_{2}^{\gamma})^{2}}\int_{0}^{r_{1}} \int_{S_{d_{2}-1}} \tilde{g}_{m}(\rho_{1},r_{2}\theta) r_{2}^{d_{2}-1} \mu(\,d\theta)\,d\rho_{1}\right] \,dr_{1}\,dr_{2}\\
		&= \int_{0}^{1}\int_{0}^{\infty} \left[\ldots\right]\,dr_{1}\,dr_{2} + \int_{1}^{\infty} \int_{0}^{\infty} \left[\ldots\right]\,dr_{1}\,dr_{2}=: I_{1} + I_{2}.
		\end{align*}
Now we estimate $I_1$ and $I_2$ separately. By a change of variables $r_{1}=zr_{2}s_{1}$ and \eqref{K8neu} we get
		\begin{align*}
			I_{1}& \leq K\int_{0}^{1} r_{2}^{d_{2}-1}\int_{0}^{\infty}\frac{\gamma r_{1}^{\gamma-1}}{\left(r_{1}^{\gamma}+z^{\gamma}r_{2}^{\gamma}\right)^{2}}\, r_{1}\,dr_{1}\,dr_{2}\\
			& = Kz^{-(\gamma-1)}\int_{0}^{1}r_{2}^{d_{2}-\gamma}\,dr_{2}\cdot\int_{0}^{\infty}\frac{\gamma s_{1}^{\gamma}}{\left(s_{1}^{\gamma}+1\right)^{2}} \,ds_{1}\\
			&\leq K_{14}z^{-(\gamma-1)}=K_{14}|s-t|^{-(\gamma-1)(\frac{1}{\alpha_{2}'}-\frac{1}{\alpha_{1}})},
		\end{align*}
since $1<\gamma<\alpha_{1}\leq 2\leq d_{2} +1$. In order to estimate $I_2$ first note that by \eqref{K8} we have
$$F(r_{1},r_{2})=\int_{|x_{1}|\leq r_{1}}\int_{\|x_{2}\|\leq r_{2}} g_{m}(x_{1},x_{2}) \,dx_{2} \,dx_{1}
			\leq \int_{|x_{1}|\leq r_{1}} g_{m}(x_{1}) \,dx_{1}\leq K_8\cdot 2r_{1}.$$
By Fubini's theorem and integration by parts with respect to $dr_{2}$ we further get
\begin{align*}
			I_{2}&= \int_{1}^{\infty} \int_{0}^{\infty}\left[ \frac{\gamma r_{1}^{\gamma-1}}{\left(r_{1}^{\gamma}+z^{\gamma}r_{2}^{\gamma}\right)^{2}}\int_{0}^{r_{1}}\int_{S_{d_{2}-1}}\tilde{g}_{m}(\rho_{1},r_{2}\theta)r_{2}^{d_{2}-1}\mu(\,d\theta)\,d\rho_{1} \right] \,dr_{1}\,dr_{2}\\
			&= -\int_{S_{d_{2}-1}}\int_{0}^{\infty}\frac{\gamma r_{1}^{\gamma-1}}{\left(r_{1}^{\gamma}+z^{\gamma}\right)^{2}} \int_{0}^{1}\int_{0}^{r_{1}} \tilde{g}_{m}(\rho_{1},\rho_{2}\theta)\rho_{2}^{d_{2}-1}\mu(\,d\theta)\,d\rho_{1}\,d\rho_{2}\,dr_{1}\,\mu(d\theta)\\
			&\phantom{=} + \int_{0}^{\infty}\int_{1}^{\infty} \frac{2\gamma^{2}z^{\gamma}r_{1}^{\gamma-1}r_{2}^{\gamma-1}}{\left(r_{1}^{\gamma}+z^{\gamma}r_{2}^{\gamma}\right)^{3}} \,F_{m}(r_{1},r_{2})\,dr_{2}\,dr_{1}\\
			&\leq \int_{0}^{\infty}\int_{1}^{\infty} \frac{2\gamma^{2}z^{\gamma}r_{1}^{\gamma-1}r_{2}^{\gamma-1}}{\left(r_{1}^{\gamma}+z^{\gamma}r_{2}^{\gamma}\right)^{3}} \,F_{m}(r_{1},r_{2})\,dr_{2}\,dr_{1}\\
			& \leq K\int_{1}^{\infty}\int_{0}^{\infty}\frac{z^{\gamma}r_{1}^{\gamma-1}r_{2}^{\gamma-1}}{\left(r_{1}^{\gamma}+z^{\gamma}r_{2}^{\gamma}\right)^{3}} \,r_{1}\,dr_{1}\,dr_{2}\\
			& = Kz^{-\gamma+1}\int_{1}^{\infty}\frac{1}{r_{2}^{\gamma}}\,dr_{2} \cdot\int_{0}^{\infty}\frac{s_{1}^{\gamma}}{\left(s_{1}^{\gamma}+1\right)^{3}}\,ds_{1}\\
			&=K_{15}z^{-\gamma+1}=K_{15}|s-t|^{-(\gamma-1)(\frac{1}{\alpha_{2}'}-\frac{1}{\alpha_{1}})},
		\end{align*}
since $\gamma>1$. Putting things together we finally get
\begin{align*}
	&\Exp\left[\|X(s)-X(t)\|^{-\gamma}\right]\leq   K |s-t|^{-\frac{\gamma}{\alpha_{1}}}\cdot \left(J_{1}+J_{2}\right)\\
	&\leq K |s-t|^{-\frac{\gamma}{\alpha_{1}}}\cdot \left( K_{14}|s-t|^{-(\gamma-1)(\frac{1}{\alpha_{2}'}-\frac{1}{\alpha_{1}})} + K_{15}|s-t|^{-(\gamma-1)(\frac{1}{\alpha_{2}'}-\frac{1}{\alpha_{1}})} \right)\\
	&\leq K |s-t|^{-\rho'},
\end{align*}

(ii) Now assume $|s-t|=mc^i\geq 1$ with $m\in[1,c)$ and $i\in\mathbb{N}_{0}$. Choose $\alpha_{2}''>\alpha_{2}$, then by Lemma \ref{specbound} we have
\begin{align*}
		\Exp\left[\|X(s)-X(t)\|^{-\gamma}\right] &= \Exp\left[\|X(mc^{i})\|^{-\gamma}\right]\\
		& \leq \Exp\left[\left(c^{i\frac{2}{\alpha_{1}}}|X^{(1)}(m)|^{2} + c^{i\frac{2}{\alpha_{2}''}}\|X^{(2)}(m)\|^{2} \right)^{-\frac{\gamma}{2}}\right]\\
		& \leq \Exp\left[\|(X^{(1)}(m),X^{(2)}(m))\|^{-\gamma}\right]\leq K_{16}<\infty
	\end{align*}
uniformly in $m\in[1,c)$ in view of Lemma \ref{negmom}, since $\gamma<2\leq 1+d_{2}$.

Combining the results of part (i) and part (ii), as above we see that \eqref{Xsigma} is fulfilled and by  Frostman's theorem we get $\dim_{\rm H}X(B)\geq \gamma$ almost surely. Since $\gamma<\alpha_{1}\dim_{\rm H}B$ is arbitrary, this concludes the proof.
\end{proof}

\subsection{Proof of our main results}

Theorem \ref{main} is now a direct consequence of Lemma \ref{upperbound} together with Lemma \ref{lowerbound} and it only remains to prove Theorem \ref{dim1}. In case $\alpha\dim_{\rm H}B \leq1$, Lemma  \ref{upperbound} and Lemma \ref{lowerbound} are still valid in the one-dimensional situation $d=1$; see Remark \ref{rem1dim}. Together these immediately give $\dim_{\rm H}X(B)=\alpha\dim_{\rm H}B=\min(\alpha\dim_{\rm H}B,1)$ almost surely. Hence it remains to prove that $\dim_{\rm H}X(B)\geq1$ almost surely if $\alpha\dim_{\rm H}B>1$, since $\dim_{\rm H}X(B)\leq1$ is obvious. But, assuming $0<\gamma<\min(\alpha\dim_{\rm H}B,1)$, we can proceed as in the proof of the upper case of Lemma \ref{lowerbound} with $E_1=1/\alpha$ and $\alpha_1'=\alpha$ to conclude that \eqref{Xsigma} holds and hence $\dim_{\rm H}X(B)\geq\min(\alpha\dim_{\rm H}B,1)$ almost surely.\hfill $\Box$

\begin{remark}
Meerschaert and Xiao \cite{MX} present an alternative analytic way to determine $\dim_{\rm H}X([0,1])$ for an operator stable L\'evy process $\{X(t)\}_{t\geq0}$ using an index theorem of Khoshnevisan et al. \cite{KXZ}. This method heavily depends on the fine structure of the exponent as given in Theorem 3.1 of Meerschaert and Veeh \cite{MV} and implicitly uses the characterization of the set $\mathcal E$ of all possible exponents as 
\begin{equation}\label{expset}
\mathcal E=E_{\rm c}+\mathfrak T\mathcal S(\mu_1)
\end{equation}
due to Holmes et al. \cite{HHM}. Here, 
$$\mathcal S(\mu_1)=\{A\in\GL(\rd):\,\mu_1(A^{-1}dx)=\mu_1(dx)\}$$
denotes the symmetry group, $\mathfrak T\mathcal S(\mu_1)$ is its tangent space and $E_{\rm c}$ is a commuting exponent with $E_{\rm c}A=AE_{\rm c}$ for every $A\in\mathcal S(\mu_1)$. For our case of an operator semistable L\'evy process, existence of a commuting exponent $E_{\rm c}$ is known by Theorem 1.11.6 in Hazod and Siebert \cite{HS}. But due to the discrete scaling it is still an open question if the set $\mathcal E$ of possible exponents has an affine representation as in \eqref{expset} with an $\mathcal S(\mu_1)$-invariant subspace. Hence it is unclear, whether the Hausdorff dimension of the range $\dim_{\rm H}X([0,1])$ of an operator semistable L\'evy process can be obtained by a generalization of the analytic approach in section 4 of Meerschaert and Xiao \cite{MX}. However, by the presented method we can additionally determine the Hausdorff dimension of the partial range $\dim_{\rm H}X(B)$ for arbitrary Borel sets $B\subseteq\rr_+$.
\end{remark}

\bibliographystyle{plain}

\begin{thebibliography}{10}

\bibitem{ARRS} Ashbaugh, M.; Rajput, B.S.; Rama-Murthy, K.; and Sundberg, C. (1992) Remarks on the positvity of densities of stable laws. {\it Probab. Math. Statist.} {\bf 13} 77--86.

\bibitem{BMS} Becker-Kern, P.; Meerschaert, M.M.; and Scheffler, H.-P. (2003) Hausdorff dimension of operator stable sample paths. {\it Monatshefte Math.} {\bf 140} 91--101.

\bibitem{BG} Blumenthal, R.M.; and Getoor, R.K. (1960) A dimension theorem for sample functions of stable processes. {\it Illinois J. Math.} {\bf 4} 370--375.

\bibitem{Cho} Chorny, V. (1987) Operator semistable distributions on $\rd$. {\it Theory Probab. Appl.} {\bf 31} 703--709.

\bibitem{Fal1} Falconer, K.J. (1985) {\it The Geometry of Fractal Sets.} Cambridge University Press, Cambridge.

\bibitem{Fal2} Falconer, K.J. (2003) {\it Fractal Geometry --  Mathematical Foundations and Applications.} 2nd Ed., Wiley, New York.

\bibitem{HS} Hazod, W.; and Siebert, E. (2001) {\it Stable Probability Measures on Euclidean Spaces and on Locally Compact Groups.} Kluwer Academic Publishers, Dordrecht.

\bibitem{Hen} Hendricks, W.J. (1973) A dimension theorem for sample functions of processes with stable components. {\it Ann. Probab.} {\bf 1} 849--853.

\bibitem{HHM} Holmes, J.; Hudson, W.; and Mason, J.D. (1982) Operator stable laws: multiple exponents and elliptical symmetry. {\it Ann. Probab.} {\bf 10} 602--612.

\bibitem{Kah} Kahane, J.-P. (1985) {\it Some Random Series of Functions.} 2nd Ed., Cambridge University Press, Cambridge.

\bibitem{KX} Khoshnevisan, D.; and Xiao, Y. (2005) L\'evy processes: capacity and Hausdorff dimension. {\it Ann. Probab.} {\bf 33} 841--878.

\bibitem{KXZ} Khoshnevisan, D.; Xiao, Y.; and Zhong, Y. (2003) Measuring the range of an additive L\'evy process. {\it Ann. Probab.} {\bf 31} 1097--1141.

\bibitem{Luc} {\L}uczak, A. (1981) Operator semi-stable probability measures on $\rr^N$. {Coll. Math.} {\bf 45} 287--300.

\bibitem{Mat} Mattila, P. (1995) {\it Geometry of Sets and Measures in Euclidean Spaces.} Cambridge University Press, Cambridge.

\bibitem{MS} Meerschaert, M.M.; and Scheffler, H.-P. (2001) {\it Limit Distributions for Sums of Independent Random Vectors.} Wiley, New York.

\bibitem{MV} Meerschaert, M.M.; and Veeh, J.A. (1993) The structure of the exponents and symmetries of an operator stable measure. {\it J. Theoret. Probab.} {\bf 6} 713--726.

\bibitem{MX} Meerschaert, M.M.; and Xiao, Y. (2005) Dimension results for sample paths of operator stable L\'evy processes. {\it Stochastic Process. Appl.} {\bf 115} 55--75.

\bibitem{Port} Port, S.C. (1968) A remark on hitting places for transient stable processes. {\it Ann. Math. Statist.} {\bf 39} 365--371.

\bibitem{PV} Port, S.C.; and Vitale, R.A. (1988) Positivity of stable densities. {\it Proc. Amer. Math. Soc.} {\bf 102} 1018--1023.

\bibitem{Pru} Pruitt, W.E. (1969) The Hausdorff dimension of the range of a process with stationary independent increments. {\it J. Math. Mech.} {\bf 19} 371--378.

\bibitem{PT} Pruitt, W.E.; and Taylor, S.J. (1969) Sample path properties of processes with stable components. {\it Z. Wahrsch. verw. Geb.} {\bf 12} 267--289.

\bibitem{Sat} Sato, K. (1999) {\it L\'evy Processes and Infinitely Divisible Distributions.} Cambridge University Press, Cambridge.

\bibitem{SW} Sato, K.; and Watanabe,. T. (2005) Last exit times for transient semistable processes. {\it Ann. Inst. H. Poincar\'e Probab. Statist.} {\bf 41} 929--951.

\bibitem{Sha} Sharpe, M. (1969) Zeroes of infinitely divisible densities. {\it Ann. Math. Statist.} {\bf 40} 1503--1505.

\bibitem{Tay} Taylor, S.J. (1967) Sample path properties of a transient stable process. {\it J. Math. Mech.} {\bf 16} 1229--1246.

\bibitem{Xiao} Xiao, Y. (2004) Random fractals and Markov processes. In: M.L. Lapidus et al. (eds.) {\it Fractal Geometry and Applications: A Jubilee of Benoit Mandelbrot}, AMS, Providence, pp. 261--338.

\bibitem{XL} Xiao, Y.; and Lin, H. (1994) Dimension properties of sample paths of self-similar processes. {\it Acta Math. Sinica} {\bf 10} 289--300.


\end{thebibliography}

\end{document}